\documentclass[12pt]{amsart}
\usepackage[margin=4cm]{geometry}
\usepackage{graphicx}
\usepackage{color}
\usepackage{enumerate}
\usepackage{tikz}
\usepackage{graphicx}
\usepackage{subcaption}
\usepackage{float}
\usepackage{comment}

\newtheorem{thm}{Theorem}[section]
\newtheorem{cor}[thm]{Corollary}
\newtheorem{lem}[thm]{Lemma}

\theoremstyle{definition}

\theoremstyle{remark}
\newtheorem{rem}[thm]{Remark}

\numberwithin{equation}{section}

\begin{document}

\title[Upper bounds for the Steklov eigenvalues of the $p$-Laplacian]{Upper bounds for the Steklov eigenvalues of the $p$-Laplacian}%
\author{Luigi Provenzano}%
\address{Luigi Provenzano, Sapienza Universit\`a di Roma, Dipartimento di Scienze di Base e Applicate per l'Ingegneria, Via Antonio Scarpa 16, 00161 Roma, Italy}%
\email{luigi.provenzano@uniroma1.it}%
\thanks{}%

\date{}%
\begin{abstract}
In this note we present upper bounds for the variational eigenvalues of the Steklov $p$-Laplacian on domains of $\mathbb R^n$, $n\geq 2$. We show that for $1<p\leq n$ the variational eigenvalues $\sigma_{p,k}$ are bounded above in terms of $k,p,n$ and $|\partial\Omega|$ only. In the case $p>n$ upper bounds depend on a geometric constant $D(\Omega)$, the {\it $(n-1)$-distortion} of $\Omega$ which quantifies the concentration of the boundary measure. We prove that the presence of this constant is necessary in the upper estimates for $p>n$ and that the corresponding inequality is sharp, providing examples of domains with boundary measure uniformly bounded away from zero and infinity and arbitrarily large variational eigenvalues.
\end{abstract}
\maketitle




\section{Introduction and statement of the main results}\label{intro}

Let $\Omega$ be a bounded domain (i.e., an open connected set) in $\mathbb R^n$, $n\geq 2$, with Lipschitz boundary $\partial\Omega$, and let $p>1$. We consider the following Steklov eigenvalue problem

\begin{equation}\label{pLap}
\begin{cases}
\Delta_pu=0\,, & {\rm in\ }\Omega,\\
|\nabla u|^{p-2}\frac{\partial u}{\partial\nu}=\sigma|u|^{p-2}u\,, & {\rm on\ }\partial\Omega,
\end{cases}
\end{equation}
where $\Delta_p={\rm div}(|\nabla u|^{p-2}\nabla u)$ is the $p$-Laplacian and $\frac{\partial u}{\partial\nu}$ is the outer normal derivative of $u$. When $p=2$ problem \eqref{pLap} is the classical Steklov problem for the Laplacian, introduced in \cite{steklov}. Problem \eqref{pLap} admits an increasing sequence of  non-negative eigenvalues, called the {\it variational eigenvalues}, diverging to $+\infty$:
$$
0=\sigma_{p,1}<\sigma_{p,2}\leq\sigma_{p,3}\leq\cdots\leq\sigma_{p,k}\leq\cdots\nearrow+\infty.
$$
A characterization of the variational eigenvalues is given by \eqref{minmax}. It it not known if the variational eigenvalues exhaust the spectrum, except in the case $p=2$.

The aim of the present note is to provide geometric upper bounds for the variational eigenvalues of problem \eqref{pLap}. Actually, being the eigenvalues not scaling invariant, we will write the upper bounds for the normalized eigenvalues, namely for $|\partial\Omega|^{\frac{p-1}{n-1}}\sigma_{p,k}$, where $|\partial\Omega|$ denotes the $(n-1)$-dimensional Hausdorff measure of $\partial\Omega$.

Upper bounds for the Steklov eigenvalues of the Laplacian have been quite extensively investigated in recent years. We recall that for a bounded Lipschitz domain of $\mathbb R^n$ the following bound holds (see \cite{colboisgirouard_steklov})
\begin{equation}\label{colbois}
|\partial\Omega|^{\frac{1}{n-1}}\sigma_{2,k}\leq\frac{C(n)}{I(\Omega)^{\frac{n-2}{n-1}}}k^{\frac{2}{n}},
\end{equation}
where $C(n)>0$ depends only on $n$ and $I(\Omega)$ denotes the isoperimetric ratio of $\Omega$ (see \eqref{iso_ratio} for the definition). In view of the Weyl's law
\begin{equation}\label{weyl_stek_2}
\lim_{k\rightarrow+\infty}\frac{|\partial\Omega|^{\frac{1}{n-1}}\sigma_{2,k}}{k^{\frac{1}{n-1}}}=\frac{2\pi}{\omega_{n-1}^{\frac{1}{n-1}}},
\end{equation}
we note that bound \eqref{colbois} does not show the expected behavior with respect to $k$, except for $n=2$. Here, by $\omega_n$ we denote the volume of the unit ball in $\mathbb R^n$. We remark that \eqref{weyl_stek_2} holds true when $\Omega$ is a piecewise $C^1$, Lipschitz domain. However, as highlighted in \cite{colboisgirouard_steklov}, a bound of the form \eqref{colbois} involving $I(\Omega)$ is not possible with a different power of $k$. Note also that \eqref{colbois} implies an upper bound on $\sigma_{2,k}$ of the form $|\partial\Omega|^{\frac{1}{n-1}}\sigma_{2,k}\leq C'(n)k^{\frac{2}{n}}$ for some constant $C'(n)>0$ depending only on $n$. Proving upper bounds of this type but with the correct exponent $\frac{1}{n-1}$ for the eigenvalue number $k$ is still an open question (except again for $n=2$). Partial results in this direction are available in \cite{stubbe_provenzano} (see also \cite{colbois_girouard_gittins} for upper bounds in the case of hypersurfaces of revolution in $\mathbb R^n$ and \cite{colbois_gittins} for upper bounds via the intersection index). We also refer to \cite{hassan} for upper bounds for the Steklov eigenvalues of the Laplacian in the conformal class of a given metric for domains in complete Riemannian manifolds.

As for the variational eigenvalues of the Steklov $p$-Laplacian, a Weyl's asymptotic law has not been established (up to our knowledge). We recall that the validity of a Weyl's law for the variational eigenvalues $\lambda_{p,k}$ of the $p$-Laplacian with Dirichlet boundary conditions on $\Omega$ of the form
\begin{equation}\label{weyl_dir}
\lim_{k\rightarrow+\infty}\frac{|\Omega|^{\frac{p}{n}}\lambda_{p,k}}{k^{\frac{p}{n}}}=C_D(p,n),
\end{equation}
with $C_D(p,n)>0$ depending only on $p$ and $n$, has been conjectured by Friedlander in \cite{friedlander_p}, who proved asymptotic upper and lower bounds for $\frac{|\Omega|^{\frac{p}{n}}\lambda_{p,k}}{k^{\frac{p}{n}}}$. The conjecture seems to have been proved recently in \cite{mazurowski}. The same discussion holds for the Neumann eigenvalues of the $p$-Laplacian. As for the Steklov eigenvalues, it in natural to conjecture that
\begin{equation}\label{weyl_stek}
\lim_{k\rightarrow+\infty}\frac{|\partial\Omega|^{\frac{p-1}{n-1}}\sigma_{p,k}}{k^{\frac{p-1}{n-1}}}=C_S(p,n),
\end{equation}
with $C_S(p,n)>0$ depending only on $p$ and $n$. Asymptotic estimates (i.e., holding for $k\geq k_{\Omega}$) in the spirit of Friedlander have been established in \cite{pinasco}.

It is reasonable to expect that upper bounds of the form \eqref{colbois} hold also for $|\partial\Omega|^{\frac{p-1}{n-1}}\sigma_{p,k}$. However, quite surprisingly, this happens only when $p\leq n$. On the other hand, for $p>n$ we show that upper bounds of the form \eqref{colbois} do not hold in general. In Section \ref{counter} we provide examples of domains $\Omega_j$ such that $|\partial\Omega_j|$ remains uniformly bounded away from zero and infinity as $j\rightarrow+\infty$, but $\lim_{j\rightarrow+\infty}\sigma_{p,2}=+\infty$. When $p>n$ we are able in any case to provide upper bounds that depend on a geometric quantity $D(\Omega)$ which we call the $(n-1)$-distortion of $\Omega$ (see \eqref{n_disto} for the definition).

We state now our main result.

\begin{thm}\label{main}
Let $\Omega$ be a bounded domain of $\mathbb R^n$ with Lipschitz boundary. Then
\begin{equation}\label{ineq_sup}
|\partial\Omega|^{\frac{p-1}{n-1}}\sigma_{p,k}\leq \frac{C_{p,n}}{I(\Omega)^{\frac{n-p}{n-1}}}k^{\frac{p}{n}}\,,\ \ \ {\rm if\ }p\leq n,
\end{equation}
where $I(\Omega)$ is the isoperimetric ratio of $\Omega$, namely
\begin{equation}\label{iso_ratio}
I(\Omega):=\frac{|\partial\Omega|}{|\Omega|^{\frac{n-1}{n}}}.
\end{equation}
Moreover
\begin{equation}\label{ineq_sub}
|\partial\Omega|^{\frac{p-1}{n-1}}\sigma_{p,k}\leq C_{p,n}'D(\Omega)^{\frac{p-n}{n-1}}k^{\frac{p-1}{n-1}}\,,\ \ \ {\rm if\ }p>n,
\end{equation}
where $D(\Omega)$ is the $(n-1)$-distortion of $\Omega$, namely
\begin{equation}\label{n_disto}
D(\Omega):=\sup_{x\in\mathbb R^n, r>0}\frac{|\partial\Omega\cap B(x,r)|}{\omega_{n-1}r^{n-1}}.
\end{equation}
The positive constants $C_{p,n},C_{p,n}'$ depend only on $p$ and $n$.
\end{thm}
We note that $D(\Omega)$ is a well-defined quantity for a bounded Lipschitz domain $\Omega$, in fact we have that $\lim_{r\rightarrow 0^+}\frac{|\partial\Omega\cap B(x,r)|}{\omega_{n-1}r^{n-1}}=0$ when $x\not\in\partial\Omega$, $\lim_{r\rightarrow+\infty}\frac{|\partial\Omega\cap B(x,r)|}{\omega_{n-1}r^{n-1}}=0$ for all $x\in\mathbb R^n$, and $\lim_{r\rightarrow 0^+}\frac{|\partial\Omega\cap B(x,r)|}{\omega_{n-1}r^{n-1}}=C(x)$ for $x\in\partial\Omega$, with $0<c\leq C(x)\leq C <+\infty$ ($C(x)=1$ for all $x\in\partial\Omega$ if $\Omega$ is of class $C^1$).

We discuss now Theorem \ref{main}.

\begin{rem}[On the case $p\leq n$]
When $p\leq n$, we note that inequality \eqref{ineq_sup} implies that a large isoperimetric ratio forces the normalized eigenvalues $|\partial\Omega|^{\frac{p-1}{n-1}}\sigma_{p,k}$ to be small when $p<n$. This is in general not true for $p=n$, at least for $p=2$ (see \cite[Theorem 4]{colbois_discr} where the authors provide an example of planar domains with large isoperimetric ratio and normalized Steklov eigenvalues bounded away from zero). We also remark that the proof of \eqref{ineq_sup} can be performed in the same way if we substitute the ambient space $\mathbb R^n$ with a complete $n$-dimensional Riemannian manifold $(M,g)$ satisfying a suitable packing property (namely, the hypothesis of Theorem \ref{gny}, see also \cite[Theorem 2.2]{colboisgirouard_steklov} for $p=2$). In particular, this is true, e.g., if $(M,g)$ has non-negative Ricci curvature. In this setting it is easier to show that we can have an arbitrarily large isoperimetric ratio and Steklov eigenvalues bounded away from zero when $p=n$. In fact, let $\Omega$ be a bounded domain in $(M,g)$, a complete $p$-dimensional Riemannian manifold as above. Let us take a conformal metric $g'=e^wg$ with $w\equiv 0$ in a neighborhood of $\partial\Omega$. The operator $\Delta_p$ is conformally covariant (recall that $p$ coincides with the space dimension), thus functions which are $p$-harmonic (i.e., with zero $p$-Laplacian) with respect to $g$ are $p$-harmonic with respect to $g'$ and vice-versa. Moreover the gradient and the normal derivative of functions along the boundary are preserved, being $g=g'$ in a neighborhood of the boundary. Therefore the Steklov  eigenvalues on $\Omega$ with respect to $g$ and $g'$ coincide. Also the measure of $\partial\Omega$ is preserved. It is sufficient then to find a function $w$ such that the volume of $\Omega$ with respect to the Lebesgue measure associated with $g'$ becomes arbitrarily small. This is done by taking some $w$ decaying rapidly to $-C$ away from $\partial\Omega$, where $C>0$ is arbitrarily large. Doing so, we obtain a very large isoperimetric ratio for $\Omega$ (in $(M,g')$), while the Steklov eigenvalues remain unchanged.
\end{rem}

\begin{rem}[On the case $p> n$]
We note that the quantity $D(\Omega)$ quantifies the concentration of the $(n-1)$-dimensional measure of $\partial\Omega$ in small regions of $\mathbb R^n$. Usually, upper bounds for the eigenvalues of Steklov-type and Neumann-type problems are not affected by the particular geometry of the domain (for the Steklov Laplacian they depend only on $k,n$ and $|\partial\Omega|$, for the Neumann Laplacian they depend only on $k,n$ and $|\Omega|$, etc.). Therefore at a first sight the geometric constant $D(\Omega)$ may result odd and unnecessary. In the case of inequality \eqref{ineq_sub} we prove in Section \ref{counter} that the constant $D(\Omega)$ is instead necessary in an upper estimate for $\sigma_{p,k}$ when $p>n$, providing a sequence of domains $\Omega_j$, $j\in\mathbb N$, with $|\partial\Omega_j|$ uniformly bounded away from zero and infinity, $\sigma_{p,2}\geq C(p,n)j^{p-n}$ as $j\rightarrow+\infty$ for some constant $C(p,n)>0$ only depending on $p$ and $n$, and with $D(\Omega_j)\sim j^{n-1}$ as $j\rightarrow+\infty$. The example not only proves the necessity of $D(\Omega)$ in \eqref{ineq_sub}, but also shows the sharpness of the exponent of $D(\Omega)$ (see Theorem \ref{thm_counter}).

We also remark that the quantity  $D(\Omega)$ already appears in some sense in the celebrated paper \cite{dancer_daners} where the authors identify a condition on the perturbations of a domain $\Omega$ under which Robin boundary conditions for the Laplacian degenerate to Dirichlet conditions at the limit. Roughly speaking, this happens when the surface measure goes locally to infinity. This is somehow equivalent to the condition that $D(\Omega)\rightarrow+\infty$. This condition, expressed in a different way, also appears in the study of Steklov-type eigenvalue problems and boundary value problems for the Laplacian on domains with very rapidly oscillating boundaries (see \cite{arrieta_bruschi_1,ferrero_lamberti,ferrero_lamberti_2}), where it implies spectral instability and degeneration of the limit problem (which, in the case of the classical Steklov problem, amounts to saying that all the eigenvalues converge to zero).  The same results in the spirit of \cite{ferrero_lamberti,ferrero_lamberti_2} are very likely to hold in the case of the Steklov $p$-Laplacian when $p\leq n$. On the other hand, in this note we observe a somehow opposite behavior for $p>n$. In fact, as already mentioned, the domains provided in Section \ref{counter} have arbitrarily large distortion and correspondingly arbitrarily large Steklov eigenvalues. Moreover, by suitably rescaling the domains, we may also assume that the boundary measure becomes arbitrarily large, along with the distortion, and still the eigenvalues remain uniformly bounded away from zero (see Remark \ref{rem_counter}).

Concerning $D(\Omega)$, we should also mention the recent paper \cite{colbois_gittins} where upper bounds for the Steklov eigenvalues of the Laplacian in terms of the intersection index  and the injectivity radius of the boundary have been obtained. In some sense, these two quantities together play the same role of $D(\Omega)$ in describing how the boundary measure accumulate. 

Finally, we mention that a behavior similar to that of our case $p>n$ has been observed for upper bounds on the Neumann eigenvalues of linear elliptic operators of order $2m$, $m\in\mathbb N$ and density on Euclidean domains (see \cite{colbois_provenzano}) and for upper bounds on Neumann eigenvalues of the $p$-Laplacian in the conformal class of a given metric in a complete Riemannian manifold (see \cite{colbois_provenzano_3}). 
\end{rem}

\begin{rem}
In view of the conjectured Weyl's law \eqref{weyl_stek}, for $p>n$ the upper bounds present the correct behavior with respect to $k$. This is somehow expected and natural since the power $k^{\frac{p}{n}}$ (which we have in the bounds for $p\leq n$) is not compatible with \eqref{weyl_stek} when $p>n$: in fact $\frac{p}{n}<\frac{p-1}{n-1}$ when $p>n$. In the case of convex domains we have that $D(\Omega)\leq \frac{n\omega_n}{\omega_{n-1}}$. Thus, when $p>n$ we have Weyl-type upper bounds for the eigenvalues (see Corollary \ref{cor_convex}).
\end{rem}

\begin{rem}[Lower bounds]
As for lower bounds, it is possible to build, for any $p>1$, a sequence of domains $\Omega_{\varepsilon}$, $\varepsilon\in(0,1)$, of fixed volume and such that $\sigma_{p,2}(\Omega_{\varepsilon})\rightarrow 0$ as $\varepsilon\rightarrow 0^+$. The construction is standard for $p=2$ (see e.g., \cite[Chapter III]{chavel}). However, for any $p>1$ it can be reproduced with no essential modifications. Namely, one considers a sequence of dumbbell domains $\Omega_{\varepsilon}={\rm Int}(\overline\Omega_1\cup\overline\Omega_2\cup\overline{\omega_{\varepsilon}})$, provided that the union is connected. Here $\Omega_1,\Omega_2$ are two disjoint bounded domains, $\omega_{\varepsilon}\sim(0,L)\times B_{\varepsilon}$, where $L>0$ and $B_{\varepsilon}$ is a ball of radius $\varepsilon$ in $\mathbb R^{n-1}$, and ${\rm Int}$ denotes the interior.  It is sufficient consider the variational characterization \eqref{minmax} of $\sigma_{2,p}$ and use as a set of test functions the set $F=\{\alpha_1u_1+\alpha_2 u_2:\alpha_1,\alpha_2\in\mathbb R, |\alpha_1|^p+|\alpha_2|^p=1\}$, where $u_i$, $i=1,2$, are functions in $W^{1,p}(\Omega)$ with $\|u_i\|_{L^p(\partial\Omega)}=1$, $u_i\equiv c_i\ne 0$ in $\Omega_i$, $u_1,u_2$ disjointly supported. It is not hard to build such test functions and use them as in the proof of Theorem \ref{main} to obtain $\sigma_{p,2}(\Omega_{\varepsilon})\leq C\varepsilon^{n-1}$ for some $C>0$ independent on $\varepsilon$.

\end{rem}

In order to prove Theorem \ref{main} we will use an approach based on a metric construction (see \cite{gny}, see also \cite{colboisgirouard_steklov}). Namely, in order to bound $\sigma_{p,k}$ we consider $A_1,...,A_k$ disjoints subsets of $\Omega$ of measure of the order of $\frac{|\Omega|}{k}$, and with $|\partial\Omega\cap \overline A_k|$ of the order of $\frac{|\partial\Omega|}{k}$, and introduce test functions $u_1,...,u_k$ subordinated to these sets. A clever estimate of the Rayleigh quotient of these functions provides the upper bounds of Theorem \ref{main}.

The paper is organized as follows. In Section \ref{pre} we set the notation and recall some preliminary results. In Section \ref{proofs} we prove Theorem \ref{main}. In Section \ref{counter} we provide the examples of domains showing the necessity of the geometric constant $D(\Omega)$ in \eqref{ineq_sub} and the sharpness of the exponent.




\section{Preliminaries and notation}\label{pre}

By $W^{1,p}(\Omega)$ we denote the Sobolev space of functions $u\in L^p(\Omega)$ with weak first derivatives in $L^p(\Omega)$. The space $W^{1,p}(\Omega)$ is endowed with the norm

\begin{equation}\label{Wnorm}
\|u\|_{W^{1,p}(\Omega)}^p:=\int_{\Omega}|\nabla u|^p+|u|^p dx,
\end{equation}
For $u\in L^p(\Omega)$ we denote by $\|u\|_{L^p(\Omega)}$ its standard norm given by
\begin{equation}\label{Lnorm}
\|u\|_{L^{p}(\Omega)}^p:=\int_{\Omega}|u|^p dx,
\end{equation}
while for $u\in L^p(\partial\Omega)$ we denote by $\|u\|_{L^p(\partial\Omega)}$ its standard norm given by
\begin{equation}\label{Lbnorm}
\|u\|_{L^{p}(\partial\Omega)}^p:=\int_{\partial\Omega}|u|^p d\sigma(x),
\end{equation}
where $d\sigma(x)$ denotes the $(n-1)$-dimensional measure element on $\partial\Omega$.

For a measurable set $E$ of $\mathbb R^n$ we denote by $|E|$ its Lebesgue measure. For subset $E$ of $\mathbb R^n$ which is measurable with respect to the $(n-1)$-dimensional Hausdorff measure, we shall still denote by $|E|$ its $(n-1)$-dimensional Hausdorff measure. Therefore, for an open set $\Omega$ of $\mathbb R^n$ with Lipschitz boundary, $|\Omega|$ shall denote its Lebesgue measure, while $|\partial\Omega|$ shall denote the $(n-1)$-dimensional measure of its boundary. By $\mathbb N$ we denote the set of positive integers.

Problem \eqref{pLap} is understood in the weak sense, namely a couple $(u,\sigma)\in W^{1,p}(\Omega)\times\mathbb R$ is a weak solution to \eqref{pLap} if and only if
\begin{equation}\label{pLap_weak}
\int_{\Omega}|\nabla u|^{p-2}\nabla u\cdot\nabla\phi dx=\sigma\int_{\partial\Omega}|u|^{p-2}u\phi d\sigma(x)\,,\ \ \ \forall\phi\in W^{1,p}(\Omega).
\end{equation}

A sequence of eigenvalues for \eqref{pLap_weak} can be obtained through the {\it Ljusternik-Schnirelman} principle (see \cite{bonder,garcia_azorero,an_le} for a more detailed discussion on the variational eigenvalues of problem \eqref{pLap}). These eigenvalues, which form an increasing sequence of non-negative numbers diverging to $+\infty$, are called the {\it variational eigenvalues} as they can be characterized variationally as follows:
\begin{equation}\label{minmax}
\sigma_{p,k}:=\inf_{F\in\Gamma_k}\sup_{u\in F}\mathcal R_p(u),
\end{equation}
where
\begin{equation}\label{Rayleigh}
\mathcal R_p(u):=\frac{\int_{\Omega}|\nabla u|^pdx}{\int_{\partial\Omega}|u|^pd\sigma(x)}
\end{equation}
is the Rayleigh quotient of $u$. Here

\begin{multline}\label{Gammak}
\Gamma_k:=\left\{F\subset W^{1,p}(\Omega)\setminus\left\{0\right\}\right.\\
\left.: F\cap\left\{u:\|u\|_{L^p(\partial\Omega)}=1 \right\}{\rm\ compact,\ }F{\rm\ symmetric,\ }\gamma(F)\geq k\right\},
\end{multline}

and $\gamma(F)$ denotes the {\it Krasnoselskii genus} of $F$, which is defined by

\begin{equation}\label{genus}
\gamma(F):=\min\left\{\ell\in \mathbb N:{\rm\ there\ exists\ }f:F\rightarrow\mathbb R^{\ell}\setminus\left\{0\right\}{\rm\ continuous\ and\ odd}\right\}.
\end{equation}

We refer to \cite{bonder} for the proof (see also \cite{garcia_azorero,an_le}).

In order to prove upper bounds for $\sigma_{p,k}$ we need suitable sets $F_k\in \Gamma_k$ to test in \eqref{minmax}. The following lemma provides us a useful way to build such $F_k$.

\begin{lem}\label{disjoint_lem}
Let $k\in\mathbb N$, $k\geq 1$, and let $u_1,...u_k\in W^{1,p}(\Omega)$, with $u_i\ne 0$ and with pairwise disjoint supports $U_1, ..., U_k$. Let
$$
F_k:=\left\{\sum_{i=1}^k\alpha_iu_i: \alpha_i\in\mathbb R, \sum_{i=1}^k|\alpha_i|^p=1\right\}.
$$
Then $F_k\in\Gamma_k$.
\end{lem}
\begin{proof}
Clearly $0\notin F_k$. Moreover, $F_k\cap\left\{u:\|u\|_{L^p(\partial\Omega)}=1 \right\}$ is compact and  $F_k$ is symmetric. We show now that $\gamma(F_k)=k$. We define a map $f_k:F_k\rightarrow\mathbb R^k\setminus\left\{0\right\}$ by setting, for $u\in F_k$, $u=\sum_{i=1}^k\alpha_i u_i$,
$$
f_k(u)=\sum_{i=1}^k\alpha_i e_i^k,
$$
where $e_i^k$, $i=1,...,k$, denotes the standard basis of $\mathbb R^k$. 
The function $f_k$ is an  odd homeomorphism between $F_k$ and $\mathbb S^{k-1}_p:=\left\{x\in\mathbb R^k:\sum_{i=1}^k|x_i|^p=1\right\}$, which is the unit sphere of $\mathbb R^k$ with respect to the $\ell^p$ norm. This implies that $\gamma(F_k)=\gamma(\mathbb S^{k-1}_p)$ (see also \cite[Proposition 2.3]{szulkin}). Finally, by the Borsuk-Ulam Theorem we deduce that $\gamma(\mathbb S^{k-1}_p)=k$.
\end{proof}

We recall now the main technical tools which will be used to prove upper bounds for eigenvalues. We denote by $(X,{\rm dist},\varsigma)$ a metric measure space with a metric ${\rm dist}$ and a Borel measure $\varsigma$. We will call {\it capacitor} every couple $(A,D)$ of Borel sets of $X$ such that $A\subset D$. By an annulus in $X$ we mean any set $A\subset X$ of the form
\begin{equation*}
A=A(a,r,R)=\left\{x\in X:r<{\rm dist}(x,a)<R\right\},
\end{equation*}
where $a\in X$ and $0\leq r<R<+\infty$. By $2A$ we denote 
\begin{equation*}
2A=2A(a,r,R)=\left\{x\in X:\frac{r}{2}<{\rm dist}(x,a)<2R\right\}.
\end{equation*}

The following theorem provides a decomposition of a metric measure space by disjoint capacitors satisfying suitable measure conditions.
\begin{thm}[{\cite[Theorem 1.1]{gny}}]\label{gny}
Let $(X,{\rm dist},\varsigma)$ be a metric-measure space with $\varsigma$ a non-atomic finite  Borel measure. Assume that the following properties are satisfied:
\begin{enumerate}[i)]
\item there exists a constant $\Gamma$ such that any metric ball of radius $r$ can be covered by at most $\Gamma$ balls of radius $\frac{r}{2}$;
\item all metric balls in $X$ are precompact sets.
\end{enumerate}
Then for any integer $k$ there exists a sequence $\left\{A_i\right\}_{i=1}^k$ of $k$ annuli in $X$ such that, for any $i=1,...,k$
\begin{equation*}
\varsigma(A_i)\geq c\frac{\varsigma(X)}{k},
\end{equation*}
and the annuli $2A_i$ are pairwise disjoint. The constant $c$ depends only on the constant $\Gamma$ in i).
\end{thm}

Theorem \ref{gny} provides a decomposition of a metric measure space by annuli of the size at least $c\frac{\varsigma(X)}{k}$. The common idea of the proof of inequalities \eqref{ineq_sup} and \eqref{ineq_sub} is to build for each $k\in\mathbb N$, suitable test functions $u_i$ supported on $2A_i$ and such that $u_i\equiv 1$ on $A_i$, and then to compute their Rayleigh quotients.

We also state a useful (but somehow hidden in the original paper \cite{gny}) corollary of Theorem \ref{gny} which gives a lower bound of the inner radius of the annuli of the decomposition, see \cite[Remark\,3.13]{gny}.
\begin{cor}\label{corollarygny0}
Let the assumptions of Theorem \ref{gny} hold. Then each annulus $A_i$ has either internal radius $r_i$ such that
\begin{equation}\label{gny-rad-est}
r_i\geq\frac{1}{2}\inf\left\{r\in\mathbb R:V(r)\geq v_k\right\},
\end{equation}
where $V(r):=\sup_{x\in X}\varsigma(B(x,r))$ and $v_k=c\frac{\varsigma(X)}{k}$ , or is a ball of radius $r_i$ satisfying \eqref{gny-rad-est}.
\end{cor}

\section{Proof of the main result}\label{proofs}

In this section we present the proof of Theorem \ref{main}.
\begin{proof}[Proof of Theorem \ref{main}]
We take the metric measure-space $(\mathbb R^n,d,\mu)$, where $d(x,y)=|x-y|$ is the Euclidean distance and the measure $\mu$ is defined by setting $\mu(E)=\int_{\partial\Omega\cap E}d\sigma(x)=|E\cap\partial\Omega|$ for an open set $E$. Note that $\mu$ is a non-atomic measure and $\mu(\mathbb R^n)=|\partial\Omega|$. It follows from Theorem \ref{gny} that, for any $k\in\mathbb N$, there exists $A_1,...,A_{2k}$ annuli in $\mathbb R^n$ with
\begin{equation}\label{cond_den}
\mu(A_i)\geq c_n\frac{\mu(\mathbb R^n)}{2k}=c_n\frac{|\partial\Omega|}{2k},
\end{equation}
and such that $2A_i$ are pairwise disjoint. The constant $c_n$ depends only on $n$. By possibly re-ordering the annuli, we have that
\begin{equation}\label{cond_num}
|2A_i\cap\Omega|\leq\frac{|\Omega|}{k}
\end{equation}
for $i=1,...,k$ (in fact we cannot have more than $k$ disjoint annuli with $|2A_i\cap\Omega|\geq\frac{|\Omega|}{k}$). Associated with each $A_i=A_i(a_i,r_i,R_i)$ we define a function $u_i$ by setting
\begin{equation}\label{ui}
u_i(x)=
\begin{cases}
1\,, & r_i\leq |x-a_i| \leq R_i,\\
\frac{2|x-a_i|}{r_i}-1\,, & \frac{r_i}{2}\leq |x-a_i|\leq r_i\,,\\
2-\frac{|x-a_i|}{R_i}\,, & R_i\leq |x-a_i|\leq 2R_i\,,\\
0\,, & {\rm otherwise}. 
\end{cases}
\end{equation}
In the case that $A_i$ is a ball of radius $r_i$ and center $a_i$, the function $u_i$ is defined by setting
\begin{equation}\label{ui2}
u_i(x)=
\begin{cases}
1\,, & |x-a_i| \leq r_i,\\
2-\frac{|x-a_i|}{r_i}\,, & r_i\leq |x-a_i|\leq 2r_i\,,\\
0\,, & {\rm otherwise}. 
\end{cases}
\end{equation}
Note that $u_i\in W^{1,p}(\Omega)$, $u_i$ is supported on $2A_i$ and $u_i\equiv 1$ on $A_i$. Let us take
$$
F_k:=\left\{\sum_{i=1}^k\alpha_iu_i:\alpha_i\in\mathbb R,\sum_{i=1}^k|\alpha_i|^p=1\right\}.
$$

From \eqref{minmax} and from Lemma \ref{disjoint_lem} we deduce that
\begin{equation}\label{minmax12}
\sigma_{p,k}\leq\sup_{u\in F_k}\mathcal R_p(u),
\end{equation}
which in particular implies, since $u_i$ are disjointly supported, that
\begin{equation}\label{minmax2}
\sigma_{p,k}\leq\max_{i=1,...,k}\mathcal R_p(u_i).
\end{equation}
Thus, in order to estimate $\sigma_{p,k}$ it is sufficient to estimate the Rayleigh quotients $\mathcal R_p(u_i)$ for $i=1,...,k$. We distinguish now the cases $p\leq n$ and $p>n$.

{\bf Case $p\leq n$.} We have, for the numerator
\begin{equation}\label{num}
\int_{\Omega}|\nabla u_i|^pdx\leq\left(\int_{\Omega}|\nabla u_i|^ndx\right)^{\frac{p}{n}}|2A_i\cap\Omega|^{1-\frac{p}{n}}\leq C_n^p\left(\frac{|\Omega|}{k}\right)^{1-\frac{p}{n}},
\end{equation}
where we have used \eqref{cond_num} and the fact that $|\nabla u_i|$ equals $\frac{2}{r_i}$ for $\frac{r_i}{2}\leq |x-a_i|\leq r_i$, $\frac{1}{R_i}$ for $R_i\leq |x-a_i|\leq 2R_i$ (and it is $\frac{1}{r_i}$ for $r_i\leq |x-a_i|\leq 2r_i$ when $A_i$ is a ball). In fact, an easy computation shows that $\left(\int_{\Omega}|\nabla u_i|^ndx\right)^{\frac{1}{n}}\leq (2n\omega_n)^{\frac{1}{n}}=:C_n$.

As for the denominator, we have

\begin{equation}\label{den}
\int_{\partial\Omega}|u_i|^pd\sigma(x)\geq\int_{A_i\cap\partial\Omega}|u_i|^pd\sigma(x)=\mu(A_i)\geq c_n\frac{|\partial\Omega|}{2k},
\end{equation}
where we have used the fact that $u_i\equiv 1$ on $A_i$ and \eqref{cond_den}. From \eqref{num} and \eqref{den} we deduce that
\begin{equation}\label{end_case1}
\frac{\int_{\Omega}|\nabla u_i|^pdx}{\int_{\partial\Omega}|u_i|^pd\sigma(x)}\leq\frac{2C_n^p}{c_n}\frac{|\Omega|^{1-\frac{p}{n}}}{|\partial\Omega|}k^{\frac{p}{n}}\leq \frac{C_{p,n}}{I(\Omega)^{\frac{n-p}{n-1}}}\frac{k^{\frac{p}{n}}}{|\partial\Omega|^{\frac{p-1}{n-1}}},
\end{equation}
where $C_{p,n}=2c_n^{-1}C_n^p$. This concludes the case $p\leq n$.

{\bf Case $p>n$.} We estimate the Rayleigh quotient of the same functions $u_i$ used in the case $p\leq n$, but in a different fashion (at least, for the numerator).   We have 
\begin{equation}\label{num2}
\int_{\Omega}|\nabla u_i|^pdx\leq\|\nabla u_i\|_{L^{\infty}(\Omega)}^{p-n}\int_{\Omega}|\nabla u_i|^ndx\leq 2n\omega_n\|\nabla u_i\|_{L^{\infty}(\Omega)}^{p-n}\leq \frac{2^{1+p-n}n\omega_n}{r_i^{p-n}}.
\end{equation}
From Corollary \ref{corollarygny0} we deduce that $r_i\geq\frac{1}{2}\inf\left\{r\in\mathbb R\: V(r)\geq c_n\frac{|\partial\Omega|}{2k}\right\}$, where  $V(r):=\sup_{x\in \mathbb R^n}\mu(B(x,r))$. From this and from the definition of $D(\Omega)$ we deduce that
\begin{equation}\label{up_ri}
r_i\geq\frac{1}{2}\left(\frac{c_n|\partial\Omega|}{2k\omega_n D(\Omega)}\right)^{\frac{1}{n-1}}.
\end{equation} 
Since for the denominator of the Rayleigh quotient the estimate \eqref{den} holds, from \eqref{den}, \eqref{num2} and \eqref{up_ri} we conclude that
\begin{equation}
\frac{\int_{\Omega}|\nabla u_i|^pdx}{\int_{\partial\Omega}|u_i|^pd\sigma(x)}\leq C_{p,n}'D(\Omega)^{\frac{p-n}{n-1}}\left(\frac{k}{|\partial\Omega|}\right)^{\frac{p-1}{n-1}},
\end{equation}
where $C_{p,n}'=2^{\frac{n(2p-2n+3)-2-p}{n-1}}c_n^{-\frac{p-1}{n-1}}n\omega_n$. This and \eqref{minmax2} allow to conclude the proof.

\end{proof}

We note that for a convex set $\mathcal D(\Omega)\leq \frac{n\omega_n}{\omega_{n-1}}$. In fact
$$
\frac{|\partial\Omega\cap B(x,r)|}{\omega_{n-1}r^{n-1}}\leq\frac{|\partial(\Omega\cap B(x,r))|}{\omega_{n-1}r^{n-1}}\leq\frac{|\partial B(x,r)|}{\omega_{n-1}r^{n-1}}=\frac{n\omega_n}{\omega_{n-1}}.
$$
We have used the fact that if $K_1,K_2$ are convex domains with $K_1\subseteq K_2$, then $|\partial K_1|\leq|\partial K_2|$. In this case $K_1=\Omega\cap B(x,r)$ and $K_2=B(x,r)$. Note that $\Omega\cap B(x,r)$ is convex being the intersection of two convex sets. We have the following corollary.

\begin{cor}\label{cor_convex}
Let $\Omega$ be a bounded and convex domain of $\mathbb R^n$ and let $p>n$. Then
\begin{equation}
\sigma_{p,k}\leq C_{p,n}''\left(\frac{k}{|\partial\Omega|}\right)^{\frac{p-1}{n-1}},
\end{equation}
where $C_{p,n}''>$ depends only on $p$ and $n$.
\end{cor}

\section{Domains with fixed surface measure and arbitrarily large variational eigenvalues}\label{counter}

The aim of this section is to build a sequence $\{\Omega_j\}_{j\in\mathbb N}\subset\mathbb R^n$ of domains which satisfy $\lim_{j\rightarrow+\infty}|\partial\Omega_j|=C>0$ and $\sigma_{p,2}(\Omega_j)\rightarrow+\infty$ when $p>n$. Through all this section we shall denote by $\sigma_{p,2}(\Omega_j)$ the second variational eigenvalue of \eqref{pLap} on $\Omega_j$. The variational eigenvalue $\sigma_{p,2}(\Omega_j)$ is actually the second eigenvalue of \eqref{pLap} (recall that $\sigma_{p,1}(\Omega_j)=0$), and every eigenfunction associated with $\sigma_{p,2}(\Omega_j)$ changes its sign on $\partial\Omega$ (see \cite{rossi_bonder_trace,martinez_rossi} for details).

Let $\alpha,\beta>0$ two positive numbers satisfying $\beta>n$, $\alpha=\beta-n+1$ (in particular, $\alpha>1$) and let $j\in\mathbb N$. Let $Q_{j}:=\left(0,\frac{1}{j}\right)^n$ be the $n$-dimensional cube in $\mathbb R^n$ of side $\frac{1}{j}$. Let now $(i_1,...,i_{n-1})\in\left\{0,...,m(j)-1\right\}^{n-1}$ with $m(j)=\left[j^{\beta-1}\right]+1$, where $[\cdot]$ denotes the integer part of a real number. Let $Q_{i_1,...,i_{n-1}}$ be the $(n-1)$-dimensional cube defined by
\begin{multline}
Q_{i_1,...,i_{n-1}}:=
\Bigg\{(x_1,...,x_n)\in\mathbb R^n
\\: x_{\ell}\in \left(c_{i_{\ell}}-\frac{1}{2jm(j)},c_{i_{\ell}}+\frac{1}{2jm(j)}\right), \ell=1,...,n-1,{\rm \ and\ }x_n=\frac{1}{j} \Bigg\},
\end{multline}
where 
$$
c_{i_{\ell}}=\frac{1}{jm(j)}\left(\frac{1}{2}+i_{\ell}\right)\,,\ \ \ \ell=1,...,n-1.
$$
Namely, the cube $Q_{1_1,...,i_{n-1}}$ has center $c_{i_1,...,i_{n-1}}$ given by
\begin{multline*}
c_{i_1,...,i_{n-1}}=\left(c_{i_1},...,c_{i_{n-1}},\frac{1}{j}\right)
\\
=\left(\frac{1}{jm(j)}\left(\frac{1}{2}+i_1\right),...,\frac{1}{jm(j)}\left(\frac{1}{2}+i_{n-1}\right),\frac{1}{j}\right)
\end{multline*}
Note that $\overline Q_j\cap\left\{x_n=\frac{1}{j}\right\}=\bigcup_{i_1,...,i_{n-1}=0}^{m(j)-1}\overline Q_{i_1,...,i_{n-1}}$. Roughly speaking, we have decomposed the upper face $\overline Q_j\cap\left\{x_n=\frac{1}{j}\right\}$ of $Q_j$ as the union of $m(j)^{n-1}\sim j^{(n-1)(\beta-1)}$  $(n-1)$-dimensional cubes of side $\frac{1}{jm(j)}\sim\frac{1}{j^{\beta}}$.

Let now $P_{i_1,...,i_{n-1}}$ be the square pyramid with base $Q_{i_1,...,i_{n-1}}$ and height $\frac{1}{j^{\alpha}}$ such that the vertex of $P_{i_1,...,i_{n-1}}$ is $(c_{i_1},...,c_{i_{n-1}},\frac{1}{j}+\frac{1}{j^{\alpha}})$. We observe that
\begin{equation}\label{vol_pyr}
|P_{i_1,...,i_{n-1}}|=\frac{1}{nj^{\alpha+n-1}m(j)^{n-1}}\sim\frac{1}{nj^{\beta n-n+1}},
\end{equation}
\begin{equation}\label{surf_pyr}
|\partial P_{i_1,...,i_{n-1}}|-|Q_{i_1,...,i_{n-1}}|=\frac{2}{(jm(j))^{(n-2)}}\cdot\sqrt{\frac{1}{j^{2\alpha}}+\frac{1}{4(jm(j))^2}}\sim\frac{2}{j^{(\beta-1)(n-1)}}
\end{equation}
and
\begin{equation}
{\rm diam}P_{i_1,...,i_{n-1}}=\sqrt{\frac{2}{4(jm(j))^2}+\frac{1}{j^{2\alpha}}}\leq\frac{\sqrt{n+4}}{2j^{\alpha}}=\frac{\sqrt{n+4}}{2j^{\beta-n+1}},
\end{equation}
where ${\rm diam}D$ denotes the diameter of a set $D$.

We finally define

$$
\Omega_j:={\rm Int}\left(\overline Q_j\cup\bigcup_{i_1,...,i_{n-1}=0}^{m(j)-1}\overline P_{i_1,...,i_{n-1}}\right),.
$$
where ${\rm Int}$ denotes the interior.  Roughly speaking, $\Omega_j$ is a $n$-dimensional cube of side $\frac{1}{j}$ with $m(j)^{n-1}\sim j^{(n-1)(\beta-1)}$ pyramids on its upper face. By construction, $\Omega_j$ is a bounded Lipschitz domain for all $j\in\mathbb N$. From \eqref{vol_pyr} and \eqref{surf_pyr} we deduce that $|\Omega_j|=\frac{1}{j^n}+o(\frac{1}{j^n})$ as $j\rightarrow +\infty$ and $|\partial\Omega_j|=2+o(2)$ as $j\rightarrow+\infty$.

\begin{figure}
\includegraphics[width=0.7\textwidth]{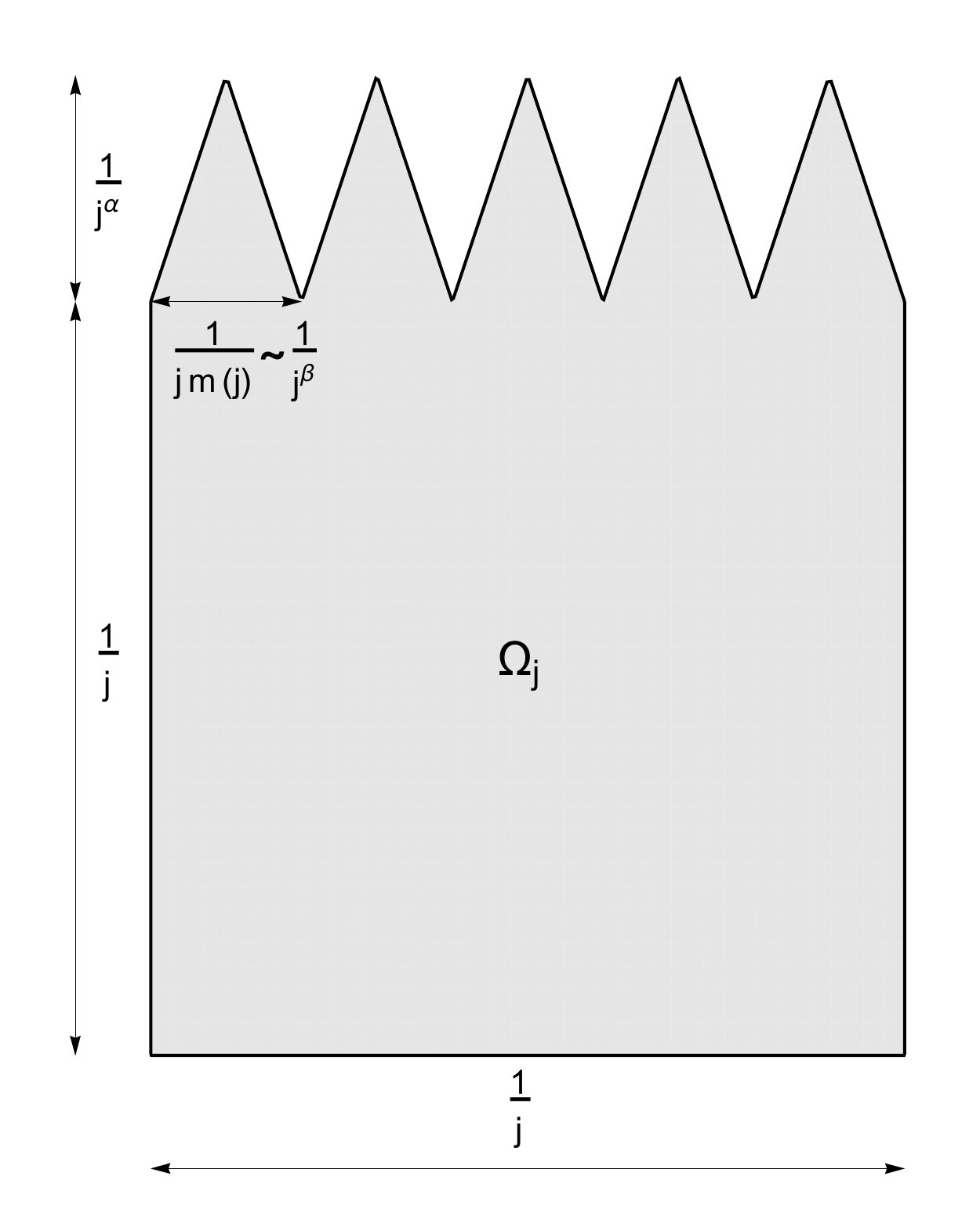}
\caption{The domain $\Omega_j$ when $n=2$.}
\end{figure}

We will prove the following theorem.

\begin{thm}\label{thm_counter}
For $p>n$ we have
\begin{equation}\label{ineq_counter}
\sigma_{p,2}(\Omega_j)\geq C(p,n) j^{p-n},
\end{equation}
where $C(p,n)>0$ depends only on $p$ and $n$.
\end{thm}

Before proving Theorem \ref{thm_counter} we need to recall a few facts on Sobolev embeddings for $p>n$. We first recall that any $u\in W^{1,p}(\Omega)$ belongs to $C^{0,\gamma}(\overline\Omega)$, for some $\gamma>0$ (or, more precisely, any $u\in W^{1,p}(\Omega)$ has a representative in $C^{0,\gamma}(\overline\Omega)$). We also recall the following lemma, the proof of which can be carried out as in \cite[Lemmas 7.12, 7.16]{gitr}.

\begin{lem}\label{morrey}
Let $\Omega$ be a bounded domain of $\mathbb R^n$, $n\geq 2$, and let $p>n$. For any convex subset $D\subset\Omega$ and any $u\in W^{1,p}(\Omega)$, we have
\begin{equation}\label{convex_morrey}
|u(x)-u(y)|\leq C'(p,n)\frac{({\rm diam}D)^n}{|D|}\cdot ({\rm diam}D)^{1-\frac{n}{p}}\|\nabla u\|_{L^p(D)},
\end{equation}
where $C'(p,n)>0$ depends only on $p$ and $n$.
\end{lem}

We are ready to prove Theorem \ref{thm_counter}

\begin{proof}[Proof of Theorem \ref{thm_counter}]
Lemma \ref{morrey} says that for any $x,y\in \overline P_{i_1,...,i_{n-1}}$,
\begin{equation}\label{convex_P}
|u(x)-u(y)|\leq C''(p,n)j^{(n-1)^2}\cdot j^{(-\beta+n-1)\left(1-\frac{n}{p}\right)}\|\nabla u\|_{L^p(\Omega_j)},
\end{equation}
where we have used the fact that $\|\nabla u\|_{L^p(P_{i_1,...,i_{n-1}})}\leq \|\nabla u\|_{L^p(\Omega_j)}$. The constant $C''(p,n)$ is strictly positive when $p>n$ and depends only on $p$ and $n$ (it can be explicitly computed, see \cite[Lemmas 7.12, 7.16]{gitr}).
We choose now $\beta=\frac{p-n(n+p-np)}{p-n}$. We easily check that $(p-n)\beta-n(p-n)=p(n-1)^2$, so that $\beta>n$. Moreover, $j^{(n-1)^2}\cdot j^{(-\beta+n-1)\left(1-\frac{n}{p}\right)}=j^{-\left(1-\frac{n}{p}\right)}$, so that \eqref{convex_P} with this choice of $\beta$ reads
\begin{equation}\label{convex_P2}
|u(x)-u(y)|\leq C''(p,n)j^{-\left(1-\frac{n}{p}\right)}\|\nabla u\|_{L^p(\Omega_j)}.
\end{equation}
Analogously, for any $x,y\in \overline Q_j$, Lemma \ref{morrey} immediately implies that
\begin{equation}\label{convex_Q}
|u(x)-u(y)|\leq C''(p,n)j^{-\left(1-\frac{n}{p}\right)}\|\nabla u\|_{L^p(\Omega_j)},
\end{equation}
where we have possibly re-defined the constant $C''(p,n)$.

From the definition of $\Omega_j$, and from \eqref{convex_P2} and \eqref{convex_Q} we deduce that for any $x,y\in\overline\Omega_j$,
\begin{equation}\label{convex_PQ}
|u(x)-u(y)|\leq 3C''(p,n)j^{-\left(1-\frac{n}{p}\right)}\|\nabla u\|_{L^p(\Omega_j)}.
\end{equation}
If furthermore we assume that there exists a point $x_0\in\overline\Omega_j$ such that $u(x_0)=0$, we immediately deduce that
\begin{equation}\label{convex_PQ_final}
|u(x)|\leq 3C''(p,n)j^{-\left(1-\frac{n}{p}\right)}\|\nabla u\|_{L^p(\Omega_j)},
\end{equation}
for any $x\in\overline\Omega_j$. We have proved that, for any $u\in W^{1,p}(\Omega_j)$ such that $u(x_0)=0$ for some $x_0\in\overline\Omega_j$
\begin{equation}\label{estimate_Rayleigh}
\frac{\int_{\Omega_j}|\nabla u|^pdx}{\int_{\partial\Omega_j}|u|^pd\sigma(x)}\geq\frac{j^{p-n}}{|\partial\Omega_j|3^pC''(p,n)^p}.
\end{equation}
We recall that any eigenfunction associated with $\sigma_{p,2}(\Omega_j)$ changes sign on $\partial\Omega_j$. This fact, the variational characterization \eqref{pLap_weak}, and \eqref{estimate_Rayleigh} allow to deduce the validity of \eqref{ineq_counter} with the constant $C(p,n)>0$ depending only on $p$ and $n$.
\end{proof}

\begin{rem}
We remark that for $\Omega_j$, $D(\Omega_j)\geq 2 j^{n-1}$, thus proving the necessity of the constant $D(\Omega)$ in an upper bound for $\sigma_{p,k}$ and the sharpness of the exponent of $D(\Omega)$ in \eqref{ineq_sub}.
\end{rem}

\begin{rem}\label{rem_counter}
When rescaling $\Omega_j$ by a factor $j^{\eta}$, we obtain 
$$
\sigma_{p,2}(\Omega_j)=j^{\eta(p-1)}\sigma_{p,2}(j^{\eta}\Omega_j).
$$
From \eqref{ineq_counter} we deduce
\begin{equation}\label{ineq_counter_2}
\sigma_{p,2}(j^{\eta}\Omega_j)\geq C(p,n) j^{-\eta(p-1)+(p-n)}.
\end{equation}
We can choose now any $0<\eta\leq\frac{p-n}{p-1}$ so that the right-hand side of \eqref{ineq_counter_2} stays bounded away from zero as $j\rightarrow+\infty$.  Note also that $|j^{\eta}\Omega_j|=j^{n(\eta-1)}+o(j^{n(\eta-1)})$ and $|\partial(j^{\eta}\Omega_j)|=2j^{\eta(n-1)}+o(j^{\eta(n-1)})$ as $j\rightarrow+\infty$. Thus $j^{\eta}\Omega_j$ has the boundary measure which goes to infinity everywhere on (a part of) the boundary as $j\rightarrow+\infty$ if $0<\eta\leq\frac{p-n}{p-1}$, but the Steklov eigenvalues remain uniformly bounded away from zero.
\end{rem}

\section*{Acknowledgements}
The author is grateful to Bruno Colbois and Pier Domenico Lamberti for fruitful discussions on the subject, and for pointing out references \cite{colbois_discr} and \cite{arrieta_bruschi_1,dancer_daners}, respectively. The author is member of the Gruppo Nazionale per le Strutture Algebriche, Geometriche e le loro Applicazioni (GNSAGA) of the I\-sti\-tuto Naziona\-le di Alta Matematica (INdAM).

\bibliographystyle{abbrv}
\bibliography{bibliography}

\end{document}